\documentclass[10pt, a4paper]{article} 

\usepackage[utf8]{inputenc} 

\usepackage{graphicx} 
\usepackage[english]{babel}
\usepackage{xcolor}

\usepackage{amsfonts,amsmath,amssymb,bbm,amsthm,amsbsy}
\usepackage[all]{xy}
\SelectTips{cm}{}

\usepackage[pagebackref,
    ,pdfborder={0 0 0}
  ,urlcolor=black,a4paper,hypertexnames=false]{hyperref}

\usepackage{color}
\usepackage{pdfcolmk}

\makeatletter
\def\blfootnote{\xdef\@thefnmark{}\@footnotetext}
\makeatother
\date{\today%
    \protect\blfootnote{\copyright{\ C.~L\"oh 2023}. 
    This work was supported by the CRC~1085 \emph{Higher Invariants} 
    (Universit\"at Regensburg, funded by the DFG).
    \\
    MSC~2010 classification: 13C05, 55N31}}

\usepackage[nouppercase]{scrpage2}
\renewcommand{\headfont}{\sffamily}  
\renewcommand{\pnumfont}{\sffamily}
\automark[subsection]{section}
\usepackage{ifthen}
\deftripstyle{myheadings}
             {\ifthenelse{\isodd{\value{page}}}{\leftmark}{\pagemark}}
             {}
             {\ifthenelse{\isodd{\value{page}}}{\pagemark}{\rightmark}}%
             {}{}{}%
\newtheorem{thm}{Theorem}[section]
\newtheorem{lemma}[thm]{Lemma}
\newtheorem{prop}[thm]{Proposition}

\theoremstyle{definition}
\newtheorem{rem}[thm]{Remark}
\newtheorem{defi}[thm]{Definition}
\newtheorem{example}[thm]{Example}

\newtheorem{algo}[thm]{Algorithm}

\newcommand{\R}{\mathbb{R}}

\newcommand{\Z}{\mathbb{Z}}
\newcommand{\N}{\mathbb{N}}

\DeclareMathOperator{\low}{low}
\DeclareMathOperator{\Span}{Span}

\DeclareMathOperator{\im}{im}

\DeclareMathOperator{\Ab}{{\sf Ab}}
\DeclareMathOperator{\Ring}{{\sf Ring}}
\DeclareMathOperator{\Mod}{{\sf Mod}}

\def\LMod#1{{}_{#1}\!\Mod}
\def\LModg#1{\LMod {#1}^*}

\def\fa#1{%
  \forall_{#1}\;\;\;}

\title{A comment on the structure of graded modules\\
  over graded principal ideal domains\\
  in the context of persistent homology}
\author{Clara L\"oh}

\begin{document}

\maketitle

\begin{abstract}
  The literature in persistent homology often
  refers to a ``structure theorem for finitely
  generated graded modules over a graded principal
  ideal domain''. We clarify the nature of
  this structure theorem in this context.
\end{abstract}

\section{Introduction}

The persistent homology with field coefficients of finite type
filtrations can be described in terms of barcodes. Zomorodian and
Carlsson promoted the elegant idea to view persistent homology with
coefficients in a field~$K$ as a graded module over the graded
polynomial ring~$K[T]$~\cite{zomorodiancarlsson}. They then suggest a
general structure theorem for finitely generated graded modules over
graded principal ideal domains~\cite[Theorem~2.1]{zomorodiancarlsson}.
Applying this structure theorem to the graded polynomial ring~$K[T]$
gives a graded elementary divisor decomposition of persistent homology,
which can be reinterpreted as barcodes~\cite{barcodes} or, equivalently,
as persistence diagrams~\cite{edelsbrunnerharer}.

However, there does not seem to be a proof of this general structure
theorem in the literature in the form stated by Zomorodian and
Carlsson. As this theorem is quoted multiple times in work 
on persistent homology and as it is a potential
source of confusion, the goal of this expository note is to clarify the nature of
this structure theorem (even though it might be clear to the experts).


We first give a precise formulation of the structure theorem; this
formulation slightly differs from the statement of Zomorodian and
Carlsson~\cite[Theorem~2.1]{zomorodiancarlsson} (for a reason
explained below):

\begin{thm}[structure theorem for graded modules over graded~PIDs]
  \label{thm:structgraded}  
  Let $R$ be a graded principal ideal domain 
  with~$R \neq R_0$ and let $M$ be a finitely generated graded
  $R$-module. Then $M$ admits a graded elementary divisor decomposition
  (Definition~\ref{def:eldiv}) and the signatures of all such
  graded decompositions of~$M$ coincide.
\end{thm}

The key observation of this note is that in fact every $\N$-graded
principal ideal domain is
\begin{itemize}
\item a principal ideal domain with the $0$-grading or
\item a polynomial ring over a field with a multiple of
  the canonical grading.
\end{itemize}
The proof is elementary~\cite[Remark~2.7]{oystaeyen} (Proposition~\ref{prop:gradedpid}).

For trivially graded principal ideal domains, in general, the graded
elementary divisor version of the structure theorem does \emph{not}
hold (Example~\ref{exa:nogradeddecomp}). This explains the additional
hypothesis of~$R \neq R_0$ in Theorem~\ref{thm:structgraded}.  In
contrast, the graded prime power version of the structure theorem also
holds if the grading is trivial (Proposition~\ref{prop:primepower0}).

For polynomial rings, the graded uniqueness part can be deduced
in a straightforward way from the ungraded uniqueness.
However, for the graded existence part, there does not seem to
be a ``generic'' derivation from the ungraded existence result --
the difficulty being the graded direct sum splitting (as exhibited
in the case of the trivially graded ring~$\Z$). Finding such
a splitting needs a careful inductive approach that establishes
that the torsion submodule is graded and that avoids
dividing out cyclic submodules in bad position/order. 
The graded existence part can be proved using specific
properties of polynomial rings over fields. 

In conclusion, the structure theorem for graded modules over graded
principal ideal domains gives a helpful structural perspective on
barcodes for persistent homology (and also for the computation of
persistent homology~\cite{zomorodiancarlsson, skrabavj}), but its scope
does not seem to go beyond the special case that is needed for
persistent homology and it does not seem to provide a shortcut
avoiding special properties of polynomial rings over fields.

Generalisations of $\N$-graded persistent homology such as zigzag
persistence or $\R$-graded persistence (or more general indexing
situations) are usually based on arguments from quiver
representations~\cite{cs_zigzag,bcb}.  Similarly to the $\N$-graded
case, in these settings, it is also essential that the underlying
coefficients are a field.


\subsection*{Organisation of this article}

Basic notions on graded rings and modules are recalled in
Section~\ref{sec:gradedbasics}. In Section~\ref{sec:gradedpid},
we prove the observation on the classification of graded
principal ideal domains (Proposition~\ref{prop:gradedpid}). The case of
principal ideal domains with trivial gradings is considered
in Section~\ref{sec:trivial}; the case of polynomial rings
over fields is discussed in Section~\ref{sec:poly}, where we
give an elementary proof of the structure theorem.

\subsection*{Acknowledgements}

I would like to thank Ulrich Bunke for helpful discussions
on abstract methods for the decomposition of graded modules
and Luigi Caputi for valuable feedback.

\section{Graded rings and modules}\label{sec:gradedbasics}

We recall basic notions on graded rings and modules
and decompositions of graded modules. As usual in
(discrete) persistence, we consider only the case
of discrete non-negative gradings, i.e., gradings over~$\N$.

\begin{defi}[graded ring]
  A \emph{graded ring} is a pair~$(R, (R_n)_{n\in\N})$, where $R$ is a
  ring and the $R_n$ are additive subgroups of~$R$ with the following
  properties:
  \begin{itemize}
  \item The additive group~$(R,+)$ is the internal direct sum of
    the~$(R_n)_{n\in \N}$.
  \item For all~$n,m \in \N$, we have~$R_n \cdot R_m \subset R_{n+m}$.
  \end{itemize}
  For $n \in \N$, the elements in~$R_n$ are called
  \emph{homogeneous of degree~$n$}. An element of~$R$ is
  \emph{homogenous} if there exists an~$n \in \N$ such that
  the element is homogeneous of degree~$n$. 

  A graded ring is a \emph{graded principal ideal domain} if it is a
  domain and every homogeneous ideal (i.e., generated by homogeneous
  elements) is generated by a single element.
\end{defi}

\begin{example}[polynomial rings]\label{exa:polygraded}
  Let $K$ be a ring. Then the usual degree on monomials in the
  polynomial ring~$K[T]$ turns~$K[T]$ into a graded ring via
  the canonical isomorphism~$K[T]
  \cong_{\Ab} \bigoplus_{n \in \N} K \cdot T^n$.  We will
  refer to this as the canonical grading on~$K[T]$. If $K$ is
  a field, then $K[T]$ is a principal ideal domain (graded
  and ungraded).
\end{example}

\begin{defi}[graded module]
  Let $R$ be a graded ring. A \emph{graded module over~$R$}
  is a pair~$(M,(M_n)_{n \in \N})$, consisting of an $R$-module~$M$
  and additive subgroups~$M_n$ of~$M$ with the following
  properties:
  \begin{itemize}
  \item The additive group~$(M,+)$ is the internal direct
    sum of the~$(M_n)_{n\in \N}$.
  \item For all~$n,m \in \N$, we have~$R_n \cdot M_m \subset M_{n+m}$.
  \end{itemize}
  Elements of~$M_m$ are called \emph{homogeneous of
    degree~$m$}.
\end{defi}

\begin{rem}[the category of graded modules]
  Let $R$ be a graded ring. \emph{Homomorphisms}
  between graded $R$-modules are $R$-linear maps that
  preserve the grading. Graded $R$-modules and graded homomorphisms
  of $R$-modules form the category~$\LModg R$ of graded $R$-modules.
\end{rem}

\begin{example}[shifted graded modules]
  Let $R$ be a graded ring, let $M$ be a graded module over~$R$, and
  let $n \in \N$. Then $\Sigma^n M$ denotes the graded $R$-module
  given by the $n$-shifted
  decomposition~$0 \oplus \dots \oplus 0 \oplus \bigoplus_{j \in \N_{\geq n}}
  M_{j-n}$.
\end{example}

\begin{example}[direct sums and quotients of graded modules]
  Let $M$ and $N$ be graded modules over a graded ring~$R$.
  Then $M\oplus N$ is a graded $R$-module via the grading~$(M_n \oplus N_n)_{n\in \N}$.
  If $M' \subset M$ is a graded submodule of~$M$ (i.e., it is generated
  by homogeneous elements), then $(M_n/(M' \cap M_n))_{n \in \N}$ turns
  $M/M'$ into a graded $R$-module.
\end{example}

Persistent homology leads to persistence
modules~\cite{zomorodiancarlsson}. Persistence modules in turn give
rise to graded modules over graded polynomial
rings~\cite[Section~3.1]{zomorodiancarlsson}:

\begin{example}[from persistence modules to graded modules]
  \label{exa:persgraded}
  Let $K$ be a ring and let $(M^*,f^*)$ be an $\N$-indexed persistence $K$-module.
  Then $M := \bigoplus_{n \in \N} M^n$ carries a $K[T]$-module structure,
  given by
  \[ \fa{x \in M^n} T \cdot x := f^n(x) \in M^{n+1}.
  \]
  If we view~$K[T]$ as a graded ring (Example~\ref{exa:polygraded}),
  then this $K[T]$-module structure and this direct sum decomposition
  of~$M$ turn~$M$ into a graded $K[T]$-module. If $(M^*,f^*)$ is
  of finite type, then $M$ is finitely generated over~$K[T]$.
\end{example}

Finally, we define the central types of decompositions arising
in the structure theorems:

\begin{defi}[graded elementary divisor decomposition]\label{def:eldiv}
  Let $R$ be a graded ring and let $M$ be a graded module over~$R$.
  A \emph{graded elementary divisor decomposition} of~$M$ over~$R$
  is an isomorphism
  \[ M \cong_{\LModg R} \bigoplus_{j=1}^N \Sigma^{n_j} R/(f_j)
  \]
  of graded $R$-modules with $N\in \N$,
  degrees~$n_1,\dots,n_N \in \N$, and 
  homogeneous elements~$f_1,\dots, f_N \in R$ with
  $f_j | f_{j+1}$ for all~$j \in \{1,\dots,N-1\}$. 
  Here, the right-hand side carries the canonical grading.
  The elements~$f_1,\dots,f_N$ are called \emph{elementary
    divisors of~$M$}.
  
  The \emph{signature} of such a decomposition is the
  multiset of all pairs~$(n_j, R^\times \cdot f_j)$
  with~$j \in \{1,\dots,N\}$. 
\end{defi}

\begin{defi}[graded prime power decomposition]
  Let $R$ be a graded ring and let $M$ be a graded module over~$R$.
  A \emph{graded prime power decomposition} of~$M$ over~$R$
  is an isomorphism
  \[ M \cong_{\LMod R} \bigoplus_{j=1}^N \Sigma^{n_j} R/(p_j^{k_j})
  \]
  of graded $R$-modules with~$N \in \N$,
  $n_1,\dots, n_N \in \N$, $k_1,\dots, k_N \in \N$,
  and homogeneous prime elements~$p_1,\dots, p_N \in R$.
  Here, the right-hand side carries the canonical grading.

  The \emph{signature} of such a decomposition is the
  multiset of all pairs~$(n_j, R^\times \cdot p_j^{k_j})$
  with~$j \in \{1,\dots,N\}$. 
\end{defi}

\section{Graded principal ideal domains}\label{sec:gradedpid}

For the sake of completeness, we provide a proof of the
following observation~\cite[Remark~2.7]{oystaeyen}.

\begin{prop}[graded PIDs]\label{prop:gradedpid}
  Let $R$ be a graded principal ideal domain.
  Then $R$ is of one of the following types:
  \begin{itemize}
  \item We have $R = R_0$, i.e., $R$ is an ordinary
    principal ideal domain with the $0$-grading.
  \item The subring~$R_0$ is  a field and $R$ is
    isomorphic to the graded ring~$R_0[T]$, where
    the grading on~$R_0[T]$ is a multiple of the
    canonical grading.
  \end{itemize}
\end{prop}
\begin{proof}
  Let $R \neq R_0$ and let $n\in \N_{>0}$ be the
  minimal degree with~$R_n \neq 0$. Then
  \[ R_{\geq n} := \bigoplus_{j \in \N_{\geq n}} R_j
  \]
  is a homogeneous ideal in~$R$; as $R$ is a graded principal
  ideal domain, there exists a~$t \in R$ with~$R_{\geq n} = (t)$.
  We show that $t$ is homogeneous of degree~$n$: Let
  $x \in R_n \setminus \{0\}$. Then $t$ divides~$x$ and
  a straightforward computation shows that hence also $t$ is
  homogeneous. 
  The grading implies that $t$ has degree~$n$.
  
  We show that the canonical $R_0$-algebra
  homomorphism~$\varphi \colon R_0[T] \longrightarrow R$
  given by~$\varphi(T) := t$ is an isomorphism.
  \begin{itemize}
  \item
    We first show that $\varphi$ is injective: Because $R$ is graded and
    $t$ is homogeneous, it suffices to show that $a \cdot t^k \neq 0$
    for all~$a \in R_0 \setminus\{0\}$ and all~$k \in \N$. However,
    this is guaranteed by the hypothesis that $R$ is a domain.
  \item
    Regarding surjectivity, let $y \in R$. It suffices to consider
    the case that $y$ is homogeneous of degree~$m\geq n$.
    Because $(t) = R_{\geq n}$, we know that $t$ divides~$y$, say~$y= t \cdot y'$.
    Then $y'$ is homogeneous and we can iterate the argument for~$y'$.
    Proceeding inductively, we obtain that $m$ is a multiple of~$n$
    and that there exists an~$a \in R_0$ with~$y = a \cdot t^{m/n}$.
    Hence, $\varphi$ is surjective.
  \end{itemize}
  This establishes that $R$ is isomorphic as a graded ring to~$R_0[T]$,
  where $R_0[T]$ carries the canonical grading on~$R_0[T]$ 
  scaled by~$n$. 

  It remains to show that $R_0 \cong_{\Ring} R/(t)$ is a field.
  Thus, we are left to show that $(t)$ is a maximal ideal in~$R$.
  By construction, every ideal~$a$ that contains~$(t) = R_{\geq n}$
  is generated by~$(t)$ and a subset of~$R_0$; in particular, $a$
  is homogeneous, whence principal. The grading shows that
  then $a = R$ or~$a= (t)$. Thus, $(t)$ is maximal and so $R_0$ 
  is a field.
\end{proof}

In the setting of $\Z$-graded principal ideal domains, further
examples appear, such as generalised Rees rings~\cite{pvg}.

\section{Trivially graded principal ideal domains}\label{sec:trivial}

\begin{example}[elementary divisor decompositions over trivially graded PIDs]
  \label{exa:nogradeddecomp}
  Let $R$ be a principal ideal domain with the $0$-grading that
  contains two non-associated prime elements~$p$ and~$q$
  (e.g., $2$ and~$3$ in~$\Z$). We consider
  the graded $R$-module
  \[ M := \Sigma^0 R/(p) \oplus \Sigma^1 R/(q).
  \]

  This graded $R$ module does \emph{not} admit a \emph{graded}
  elementary divisor decomposition: Indeed, if there were a graded
  elementary divisor decomposition of~$M$, then the corresponding
  elementary divisors would have to coincide with the ungraded
  elementary divisors. The only ungraded elementary divisor of~$M$
  is~$p\cdot q$.  However, $M$ does \emph{not} contain a homogenous
  element with annihilator ideal~$(p \cdot q)$. Therefore, $M$ does
  not admit a graded elementary divisor decomposition.
\end{example}

\begin{prop}[prime power decompositions over trivially graded PIDs]
  \label{prop:primepower0}
  Let $R$ be a principal ideal domain with the $0$-grading and let
  $M$ be a finitely generated graded $R$-module. Then $M$ admits
  a graded prime power decomposition and the signature of
  all such graded decompositions of~$M$ coincide.
\end{prop}
\begin{proof}
  Because $R$ is trivially graded, the grading on~$M$
  decomposes $M$ as a direct sum~$\bigoplus_{n \in \N} M_n$
  of $R$-submodules. In view of finite generation of~$M$,
  only finitely many of these summands are non-trivial.
  We can now apply the ungraded structure theorem
  to each summand~$M_n$ to conclude.
\end{proof}

\section{Polynomial rings over fields}\label{sec:poly}

In view of Proposition~\ref{prop:gradedpid}, Theorem~\ref{thm:structgraded}
can equivalently be stated as follows (which is exactly the special
case needed in persistent homology):

\begin{thm}[structure theorem for graded modules over polynomial rings]
  \label{thm:structgradedpoly}
  Let $K$ be a field and let $M$ be a finitely generated graded module
  over the graded ring~$K[T]$. Then there exist~$N \in \N$, $n_1,\dots, n_N \in \N$,
  and $k_1, \dots, k_N \in \N_{>0} \cup \{\infty\}$ with
  \[ M \cong_{\LModg{K[T]}} \bigoplus_{j=1}^N \Sigma^{n_j} K[T]/(T^{k_j}).
  \]
  Here, $T^\infty := 0$.
  The multiset of all~$(n_j,k_j)$ with~$j \in \{1,\dots,N\}$ is uniquely
  determined by~$M$.
\end{thm}

The rest of this section contains an elementary and constructive proof of
Theorem~\ref{thm:structgradedpoly}.

\subsection{Uniqueness of graded decompositions}\label{subsec:unique}

The uniqueness claim in Theorem~\ref{thm:structgradedpoly} can be
derived inductively from the ungraded uniqueness statement: 

Let a decomposition as in Theorem~\ref{thm:structgradedpoly} be
given and let $\varphi \colon \bigoplus_{\dots} \dots \longrightarrow M$
be a corresponding graded $K[T]$-isomorphism. Then
\[ M' := \varphi(N') \text{ with } N' := \bigoplus_{j\in \{1,\dots,N\}, n_j = 0} \Sigma^{n_j} K[T]/(T^{k_j})
\]
is a graded submodule of~$M$ and it is not difficult to see
that~$M' = \varphi(N') = \Span_{K[T]} M_0$.
Moroever, $M'$ is finitely generated over~$K[T]$. 
Therefore, the ungraded structure theorem when applied to~$M'$ shows
that the multiset of all pairs~$(n_j,k_j)$ with~$n_j=0$ is uniquely
determined by~$M$.

For the induction step,
we pass to the quotient~$M/M'$, which is a finitely generated
graded $K[T]$-module with~$(M/M')_0 \cong 0$. We shift the degrees
on~$M/M'$ by~$-1$ and inductively apply the previous argument.

\subsection{Homogeneous matrix reduction}\label{subsec:reduce}

The standard matrix reduction algorithm for the computation of
persistent homology~\cite{edelsbrunnerharer,zomorodiancarlsson} can be
viewed as a proof of the existence part of
Theorem~\ref{thm:structgradedpoly}. 

We phrase the matrix reduction algorithm in the graded
language to emphasise the connection with graded decompositions.

\begin{defi}[graded matrix]
  \label{def:matrixgraded}
  Let $K$ be a field, let $r, s \in \N$, and let $n_1, \dots, n_r$, $m_1, \dots, m_s
  \in \N$ be two monotonically increasing sequences. A matrix~$A \in M_{r \times s} (K[T])$ is
  \emph{$(n_*,m_*)$-graded} if the following holds:
  For all~$j \in \{1,\dots,r\}, k \in \{1,\dots,s\}$, we have
  that the entry~$A_{jk} \in K[T]$ is a homogeneous polynomial and 
  \begin{itemize}
  \item $A_{jk} = 0$ or
  \item $m_k = n_j + \deg A_{jk}$.
  \end{itemize}
\end{defi}

In a graded matrix, the degrees of matrix entries monotonically
increase from the left to the right and from the bottom to the top.

\begin{defi}[reduced matrix]
  Let $K$ be a field, let $r, s \in \N$, let $n_1, \dots, n_r$ 
  and $m_1, \dots, m_s \in \N$ be two monotonically increasing sequences,
  and let $A \in
  M_{r \times s}(K[T])$ be an $(n_*,m_*)$-graded matrix.
  \begin{itemize}
  \item For~$k \in \{1,\dots, s\}$, we define
    \[ \low_A(k) := \max \bigl\{ j \in \{1,\dots,r\} \bigm| A_{jk} \neq 0\bigr\}
    \in \N 
    \]
    (with~$\max \emptyset := 0$). I.e., $\low_A(k)$ is
    the index of the ``lowest'' matrix entry in column~$k$ that is non-zero.
  \item The matrix~$A$ is \emph{reduced} if all columns have
    different $\low$-indices:
    For all~$k,k' \in \N$ with~$\low_A(k) \neq 0$ and $\low_A(k') \neq 0$,
    we have~$\low_A(k) \neq \low_A(k')$.
  \end{itemize}
\end{defi}

Graded matrices can be transformed into reduced matrices via
elementary column operations; these reduced matrices then lead to
module decompositions:

\begin{algo}[homogeneous matrix reduction]
  \label{algo:reduct}
  Given a field~$K$, $r, s \in \N$, monotonically increasing
  sequences~$n_1, \dots, n_r$ and $m_1, \dots, m_s
  \in \N$, and  an $(n_*,m_*)$-graded
  matrix~$A \in M_{r \times s}$, do the following: 
  \begin{itemize}
  \item
    For each~$k$ from~$1$ up to~$s$ (in ascending order):

    Let $\ell := \low_A(k)$.

    If $\ell \neq 0$, then:
    \begin{itemize}
    \item For each~$j$ from~$\ell$ down to~$1$ (in descending order):

      If $A_{jk} \neq 0$ and there exists~$k' \in \{1,\dots,k-1\}$
      with~$\low_A(k') = j$, then:
      \begin{itemize}
      \item 
        Update the matrix~$A$ by subtracting~$A_{jk}/A_{jk'}$-times
        the column~$k'$ from column~$k$.
    
        [Loop invariant observation: Because $A$ is graded,
          $A_{jk}/A_{jk'}$ indeed is a homogeneous polynomial over~$K$
          and the resulting matrix is $(n_*,m_*)$-graded.  This eliminates the
          entry~$A_{jk'}$.]
      \end{itemize}
    \end{itemize}
    \item Return the resulting matrix~$A$.
  \end{itemize}
\end{algo}

\begin{prop}
  \label{prop:algoreduced}
  Let $K$ be a field, 
  let $r, s \in \N$, let $n_1, \dots, n_s$ and $m_1, \dots, m_r
  \in \N$ be monotonically increasing,
  and let $A \in M_{r \times s}(K[T])$ be an $(n_*,m_*)$-graded
  matrix. Then:
  \begin{enumerate}
  \item The homogeneous matrix reduction algorithm
    (Algorithm~\ref{algo:reduct})
    terminates on this
    input after finitely many steps (relative to
    the arithmetic on~$K$).
  \item The resulting matrix~$A'$ is reduced and there
    is a graded $s \times s$-matrix~$B$ over~$K[T]$
    that admits a graded inverse and satisfies
    $A' = A \cdot B.
    $ 
  \item The low-entries of the resulting matrix~$A'$
    are the elementary divisors of~$A$ over~$K[T]$.
  \item We have
    \[ F/\im A \cong_{\LModg {K[T]}}
    \bigoplus_{j \in I} \Sigma^{n_{j}} K[T]/(T^{m_{k(j)} - n_{j}})
    \oplus
    \bigoplus_{j \in I'} \Sigma^{n_j} K[T],
    \]
    where $F := \bigoplus_{j=1}^r \Sigma^{n_j} K[T]$ and 
    $I := \{ \low_{A'}(k) \mid k \in \{1,\dots,s\}\} \setminus \{0\}$
    as well as $I' := \{1,\dots,r\} \setminus I$.
    For~$j \in I$, let~$k(j) \in \{1,\dots,s\}$ be the unique~(!)
    index with~$\low_{A'} (k(j)) = j$.
  \end{enumerate}
\end{prop}
\begin{proof}
  \emph{Ad~1.}
  Well-definedness follows from the observation
  mentioned in the algorithm: As every homogeneous
  polynomial in~$K[T]$ is of the form~$\lambda \cdot T^d$
  with~$\lambda \in K$ and $d \in \N$ and
  as the matrix is graded, the corresponding division
  can be performed in~$K[T]$ and the gradedness of
  the matrix is preserved
  by the elimination operation. 
  Termination is then clear from the algorithm.

  \emph{Ad~2.}
  As we traverse the columns from left
  to right, a straightforward induction
  shows that no two columns can remain that have
  the same non-zero value of~``$\low_A$''.
  The product decomposition comes from
  the fact that we only applied elementary
  homogeneous column operations without swaps.

  \emph{Ad~3.}
  Because the resulting matrix~$A'$ is obtained through
  elementary column operations from~$A$, the
  elementary divisors of~$A'$ and $A$ coincide.
  Applying Lemma~\ref{lem:reducedeldiv} to~$A'$
  proves the claim.

  \emph{Ad~4.}
  In view of the second part, we have
  that $F / \im A \cong_{\LModg {K[T]}} F / \im A'$.
  Therefore, the claim is a direct consequence of 
  Lemma~\ref{lem:reducedeldiv}.
\end{proof}

\begin{lemma}\label{lem:reducedeldiv}
  Let $K$ be a field, let $r, s \in \N$, let $n_1, \dots, n_r$ and $m_1,
  \dots, m_s \in \N$ be monotonically increasing, and let $A \in M_{r
    \times s}(K[T])$ be an $(n_*,m_*)$-graded matrix that is reduced.
  Then: 
  \begin{enumerate}
  \item The $\low$-entries of~$A$ are the elementary divisors
    of~$A$ over~$K[T]$.
  \item Let $F := \bigoplus_{j=1}^r \Sigma^{n_j} K[T]$ and
    $I := \{ \low_A(k) \mid k \in \{1,\dots,s\}\} \setminus \{0\}$
    as well as $I' := \{1,\dots,r\} \setminus I$.
    Then
    \[ F/\im A \cong_{\LModg {K[T]}}
    \bigoplus_{j \in I} \Sigma^{n_j} K[T]/(T^{m_{k(j)} - n_j})
    \oplus
    \bigoplus_{j \in I'} \Sigma^{n_j} K[T]
    \]
  \end{enumerate}
\end{lemma}
\begin{proof}
  \emph{Ad~1.} 
  Let $k \in \{1,\dots,s\}$ with~$\ell := \low_A(k) \neq 0$. Then
  we can clear out all the entries of~$A$ in
  column~$k$ above~$\ell$
  by elementary row operations (again, the gradedness of~$A$
  ensures that this is possible). 
  Swapping zero rows and columns appropriately thus
  results in a matrix in rectangle ``diagonal''
  form; moreover, as all the ``diagonal'' entries are
  monomials, we can swap rows and columns to obtain
  a matrix~$A'$ in Smith normal form
  that both
  \begin{itemize}
  \item has the same elementary divisors as~$A$ and
  \item whose elementary divisors are precisely the $\low$-entries
    of~$A$.
  \end{itemize}
  In particular, these elementary divisors must coincide.
  
  \emph{Ad~2.}
  The claim is clear if $A$ is already in Smith normal
  form. 
  By construction, there are square matrices~$B$
  and~$C$ that are invertible over~$K[T]$ and represent
  graded $K[T]$-isomorphisms with
  \[ A' = C \cdot A \cdot B.
  \]
  In particular, $F/\im A \cong_{\LModg{K[T]}} (C \cdot F)/\im A'$. 
  By construction, the values of~$\low_{A'}$ and the degrees of~$A'$
  differ from the ones of~$A$ only by compatible index permutations.
  Therefore, the claim follows.
\end{proof}

\subsection{Existence of a graded decomposition}\label{subsec:ex}

To prove existence in Theorem~\ref{thm:structgradedpoly}
we can follow the standard proof pattern of first finding
a (graded) finite presentation and then applying (homogeneous)
matrix reduction.

Let $M$ be a finitely generated graded $K[T]$-module.
Then $M$ also has a finite generating set consisting
of homogeneous elements.
This defines a surjective graded $K[T]$-homomorphism
\[ \varphi \colon  F := \bigoplus_{j=1}^r \Sigma^{n_j} K[T] \longrightarrow M 
\]
for suitable~$r \in \N$ and monotonically increasing~$n_1,\dots,n_r \in \N$.
As $\varphi$ is a graded homomorphism, $\ker \varphi \subset F$
is a graded $K[T]$-submodule
and we obtain an isomorphism
\[ M \cong_{\LModg{K[T]}} F/\im \ker \varphi
\]
of graded $K[T]$-modules.

Because $K[T]$ is a principal ideal domain, the graded submodule~$\ker
\varphi \subset F$ is finitely generated over~$K[T]$. Because $\ker \varphi$
is a graded submodule, $\ker \varphi$ has a finite homogeneous generating
set. (In fact, there also exists a homogeneous free $K[T]$-basis
for~$\ker \varphi$, as can be seen from a straightforward
inductive splitting argument~\cite[Lemma~1]{webb}.)
In particular,
there exist~$s \in \N$, monotonically increasing~$m_1,\dots,m_s \in \N$,
and a graded $K[T]$-homomorphism
\[ \psi \colon E := \bigoplus_{k=1}^s \Sigma^{m_k} K[T] \longrightarrow F
\]
with~$\im \psi = \ker \varphi$. Because $\psi$ is graded and $n_*$,
$m_*$ are monotonically increasing, the $r \times s$-matrix~$A$
over~$K[T]$ that represents~$\psi$ with respect to the canonical
homogeneous bases of $E$ and~$F$ is graded in the sense of
Definition~\ref{def:matrixgraded}.

Applying the homogeneous matrix reduction
algorithm to~$A$ shows that
\[ M \cong_{\LModg {K[T]}} F/\im A,
\]
has the desired decomposition
(Proposition~\ref{prop:algoreduced}; after discarding
the irrelevant terms of the form~$\Sigma^nK[T]/(T^0)$).

This completes the proof of the structure theorem
(Theorem~\ref{thm:structgradedpoly}).

\begin{rem}
  There is a general matrix reduction for a slighlty different notion
  of ``graded'' matrices over ($\Z$-)graded principal ideal
  domains~\cite{pvg}. However, one should be aware that such
  ``graded'' matrices in general only lead to graded homomorphisms
  once one is allowed to change the grading on the underlying free
  modules.  This explains why this general matrix reduction does not
  contradict the counterexample in the case of $0$-graded principal ideal
  rings in Example~\ref{exa:nogradeddecomp}.
\end{rem}

\subsection{Barcodes}

For the sake of completeness, we recall the relation
between graded decompositions and barcodes:

\begin{rem}[barcodes of persistence modules]
  Let $K$ be a field and let $(M^*,f^*)$ be an $\N$-indexed persistence
  $K$-module of finite type. We equip~$M := \bigoplus_{n \in \N} M^n$
  with the canonical graded $K[T]$-module structure
  (Example~\ref{exa:persgraded}). By the graded
  structure theorem (Theorem~\ref{thm:structgradedpoly}),
  there exist~$N \in \N$, $n_1, \dots, n_N \in \N$,
  and $k_1,\dots, k_N \in \N_{>0} \cup \{\infty\}$ with
  \[ M \cong_{\LModg {K[T]}} \bigoplus_{j=1}^N \Sigma^{n_j}K[T]/(T^{k_j}).
  \]
  Let $B$ be the multiset 
  of all~$(n_j,k_j - 1)$ with~$j \in \{1,\dots,N\}$; then $B$ 
  is uniquely determined by~$M$ and this multiset~$B$ is
  the \emph{barcode of~$(M^*,f^*)$}.
  
  The barcode contains the full information on the isomorphism
  type of the graded $K[T]$-module~$M$ (and the underlying persistence
  module) and describes the birth,
  death, and persistence of elements as specified by the ``elder
  rule'': If $(n,p)$ is an element of the barcode, this means
  that a new independent class is born at stage~$n$, it persists
  for $p$~stages, and it dies (if~$p \neq \infty$) at stage~$n+p+1$.

  In particular, this leads to the notion of barcodes of
  persistent homology (in a given degree) of finite type
  persistence chain complexes and finite type filtrations in topology.
\end{rem}

{\small
\bibliographystyle{alpha}
\bibliography{bib}}

\vfill

\noindent
\emph{Clara L\"oh}\\[.5em]
  {\small
  \begin{tabular}{@{\qquad}l}
    Fakult\"at f\"ur Mathematik,
    Universit\"at Regensburg,
    93040 Regensburg\\
    \textsf{clara.loeh@mathematik.uni-r.de}, 
    \textsf{https://loeh.app.ur.de}
  \end{tabular}}

\end{document}